\documentclass[12pt]{amsart}
\usepackage[margin=1in]{geometry}
\usepackage[hidelinks]{hyperref}
\usepackage{breakurl}

\usepackage{amsmath,amsthm,amssymb}

\usepackage{tikz}
\usepackage{tikz-cd}
\usepackage{graphicx}
\usepackage{caption}
\usepackage{subcaption}
\usepackage{float}
\usepackage[export]{adjustbox}

\usepackage{booktabs}  
\usepackage{mathtools} 
\usepackage{units} 
\usepackage{ytableau} 
\usepackage{biblatex}  
\theoremstyle{plain}
\newtheorem{theorem}{Theorem}
\newtheorem{proposition}{Proposition}[section]

\theoremstyle{definition}
\newtheorem{definition}[proposition]{Definition}

\newtheorem{example}[proposition]{Example}
\newtheorem{remark}[proposition]{Remark}

\newcommand{\abs}[1]{\left|#1\right|}              

\title[A Bijection between Strongly Stable and Totally Symmetric Partitions]{A Bijection Between Strongly Stable and Totally Symmetric Partitions}
\author{Seth Ireland}
\address{Dept.~of Mathematics, Colorado State University, Fort Collins, CO, USA}

\email{seth.ireland@colostate.edu}


\begin{document}

\begin{abstract}
Artinian monomial ideals in $d$ variables correspond to $d$-dimensional partitions. We define $d$-dimensional strongly stable partitions and show that they correspond to strongly stable ideals in $d$ variables. We then show a bijection between strongly stable partitions and totally symmetric partitions which preserves the side length of the minimal bounding box. 
\end{abstract}

\maketitle

\section{Introduction}

The subgroup of $GL_d(K)$ consisting of upper triangular matrices is called the Borel group. Ideals which are fixed by the Borel group are monomial ideals called Borel-fixed ideals. Generic initial ideals are always Borel-fixed, and in characteristic zero, ideals are Borel-fixed if and only if they are strongly stable \cite{herzog2011monomial}. So the generic initial ideals over a field of characteristic zero are exactly the strongly stable ideals \cite{francisco2011borel}.

Artinian monomial ideals in $K[x_1,\dots,x_d]$ correspond naturally to $d$-dimensional partitions by considering the monomials which are not in the ideal. We can think of these as $d$-dimensional stacks of blocks in a ($d$-dimensional) corner. Totally symmetric partitions are $d$-dimensional partitions which are fixed by the symmetric group $S_d$.

In this paper, we define strongly stable partitions (Definition 2.5), show that they naturally correspond to strongly stable ideals (Proposition 4.5), and finally show a bijection with totally symmetric partitions which preserves the side length $n$ of the smallest box containing the partitions (Theorem 1).

\section{Partitions}

Let $d$ be a positive integer and let $\mathbb{N}=\{0,1,2,3,\dots\}$.

\begin{definition}
A \textit{$d$-dimensional partition} is a finite subset $P\subset\mathbb{N}^d$ such that 
\begin{equation*}
(a_1,a_2,\dots,a_j,\dots,a_d)\in P\implies (a_1,a_2,\dots,a_j-1,\dots,a_d)\in P
\end{equation*}
for all $1\leq j\leq d$ if $a_j>0$. Refer to the elements of a partition as \textit{cells}.
\end{definition}

Denote the set of $d$-dimensional partitions with largest coordinate of any cell equal to $n-1$ by $\mathcal{P}_d(n)$. We think of these as $d$-dimensional partitions which fit inside a $d$-dimensional box of side length $n$ and also touch at least one edge of the box. Note that the "boxes" we refer to are $n\times\cdots\times n$ boxes of dimension $d$ (every side of the box is the same length). Integer partitions correspond exactly with the 2-dimensional partitions.

\begin{example}
For example, consider the integer partition 7=3+2+1+1.
We can represent the integer partition with the Ferrers diagram below,
\begin{center}
\begin{ytableau}
    \none & \\
    \none & & \\
    \none & & & &
\end{ytableau}
\end{center}
which depicts the 2-dimensional partition
\begin{equation*}
P=\{(0,0),(0,1),(0,2),(1,0),(1,1),(2,0),(3,0)\}\in\mathcal{P}_2(4).
\end{equation*}
\end{example}

\begin{example}
A 3-dimensional partition is commonly called a \textit{plane partition}. An example of a plane partition of 10 is 
\begin{equation*}
P=\{(0,0,0),(0,0,1),(0,1,0),(0,1,1),(1,1,0),(1,0,1),(1,2,0),(0,2,0),(2,0,0)\}\in\mathcal{P}_3(3)
\end{equation*}
which we can visualize with the diagram below
\begin{center}
\includegraphics[scale=0.3]{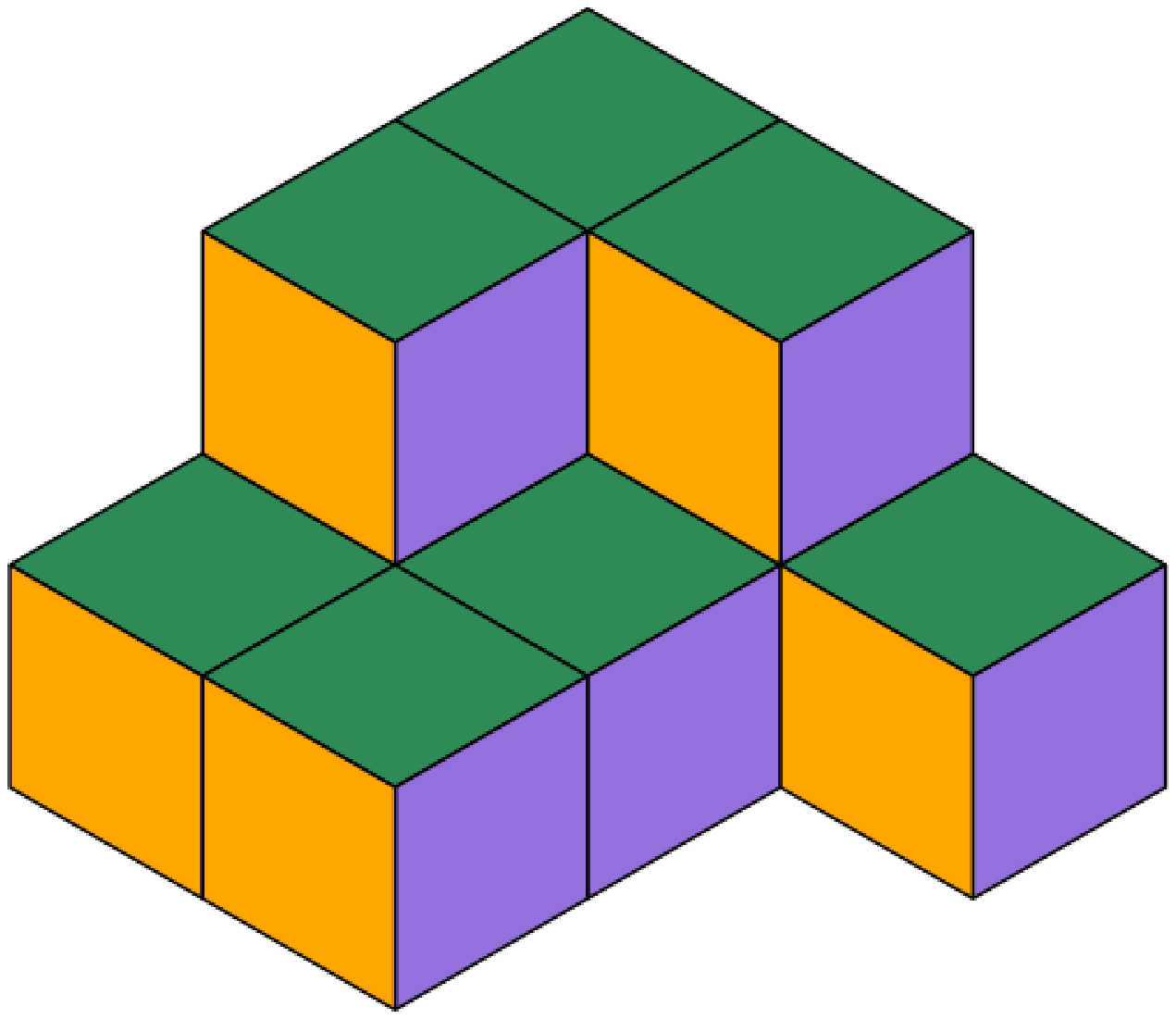}
\end{center}
Plane partitions are also commonly represented in matrix notation. In matrix notation, the partition above could be written
\begin{equation*}
\begin{matrix}
2 & 2 & 1\\
2 & 1\\
1 & 1
\end{matrix}
\end{equation*}
\end{example}

\begin{definition}
Let $P$ be a $d$-dimensional partition. For a cell $\alpha=(a_1,\dots,a_d)\in P$, define the \textit{$j$th arm length} to be the largest integer $h_j$ such that $\alpha=(a_1,a_2,\dots,a_j+h_j,\dots,a_d)\in P$. Denote the vector of arm lengths by $H(\alpha)=(h_1,\dots,h_d)$ and call this the \textit{Hook vector} of $\alpha$.
\end{definition}

\begin{definition}
A \textit{strongly stable partition} is a partition for which every cell's Hook vector is weakly increasing.
\end{definition}

We will denote the set of strongly stable partitions which fit inside a box of side length $n$ by
\begin{equation*}
\tilde{\mathcal{B}}_d(n):=\{P\mid P\in\mathcal{P}_d(n)\text{ and $P$ is strongly stable}\}
\end{equation*}

\begin{example}
An example of a 2-dimensional strongly stable partition $P\in\tilde{\mathcal{B}}_2(7)$ is given below with Hook vectors given inside each cell.
\begin{center}
\begin{ytableau}
    \none & \text{\tiny 0,0} \\
    \none & \text{\tiny 0,1}\\
    \none & \text{\tiny 0,2}\\
    \none & \text{\tiny 1,3} &\text{\tiny 0,0} \\
    \none & \text{\tiny 2,4} & \text{\tiny 1,1} & \text{\tiny 0,0}\\
    \none & \text{\tiny 2,5} & \text{\tiny 1,2} & \text{\tiny 0,1}\\
    \none & \text{\tiny 3,6} & \text{\tiny 2,3} & \text{\tiny 1,2} & \text{\tiny 0,0}
\end{ytableau}
\end{center}
Note that 2-dimensional strongly stable partitions are exactly the integer partitions with no repeats. These are commonly called \textit{strict} partitions.
\end{example}

\begin{example}
An example of a 3-dimensional strongly stable partition $P\in\tilde{\mathcal{B}}_3(4)$ is given below. One can verify that the Hook vectors of each cell are weakly increasing.
\begin{center}
\includegraphics[scale=0.4]{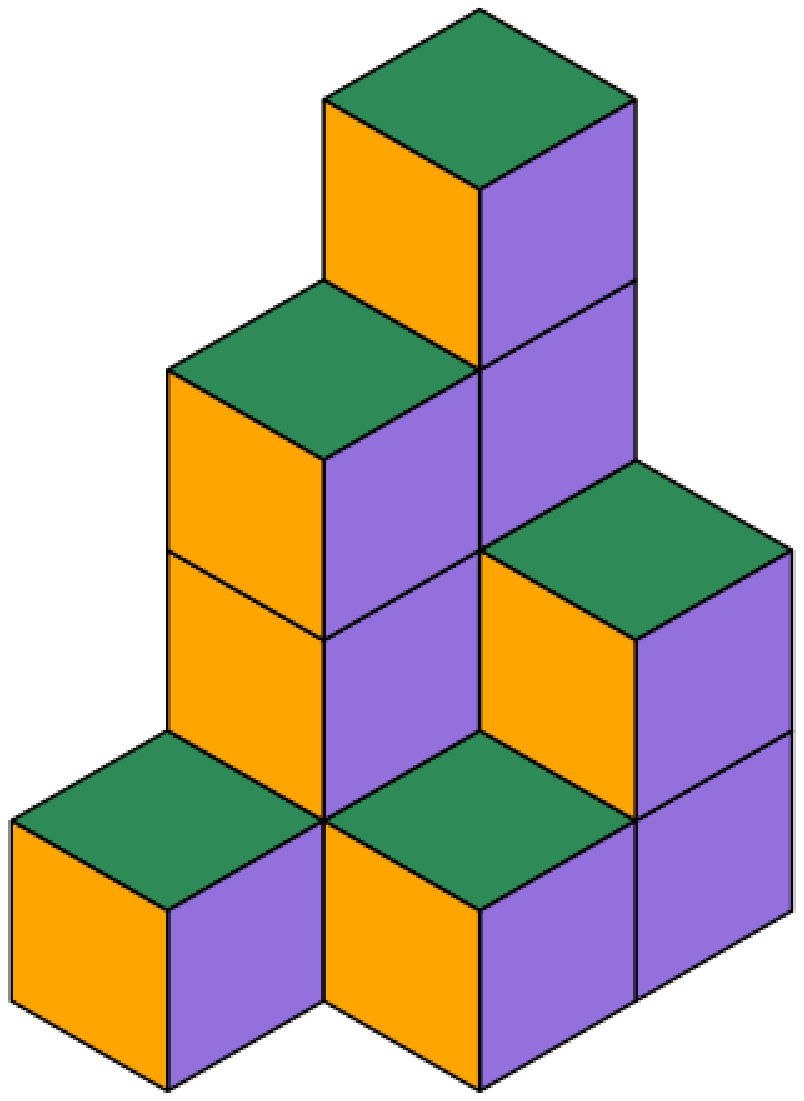}
\end{center}
\end{example}

\begin{definition}
A \textit{totally symmetric partition} is a partition for which $\alpha\in P\implies\pi(\alpha)\in P$ for all $\pi\in S_d$.
\end{definition}

We denote the set of totally symmetric partitions which fit inside a box of side length $n$ by \begin{equation*}
\tilde{\mathcal{T}}_d(n):=\{P\mid P\in\mathcal{P}_d(n)\text{ and $P$ is totally symmetric}\}
\end{equation*}

\begin{example}
An example of a 2-dimensional totally symmetric partition $P\in\tilde{\mathcal{T}}_2(7)$ is given below.
\begin{center}
\begin{ytableau}
    \none & \\
    \none & \\
    \none & & &\\
    \none & & & & \\
    \none & & & & & \\
    \none & & & & &\\
    \none & & & & & & &
\end{ytableau}
\end{center}
Note that 2-dimensional totally symmetric partitions are commonly called \textit{self-conjugate} partitions.
\end{example}

\begin{example}
An example of a 3-dimensional totally symmetric partition (totally symmetric plane partition) $P\in\tilde{\mathcal{T}}_3(4)$ is shown below.
\begin{center}
\includegraphics[scale=0.5]{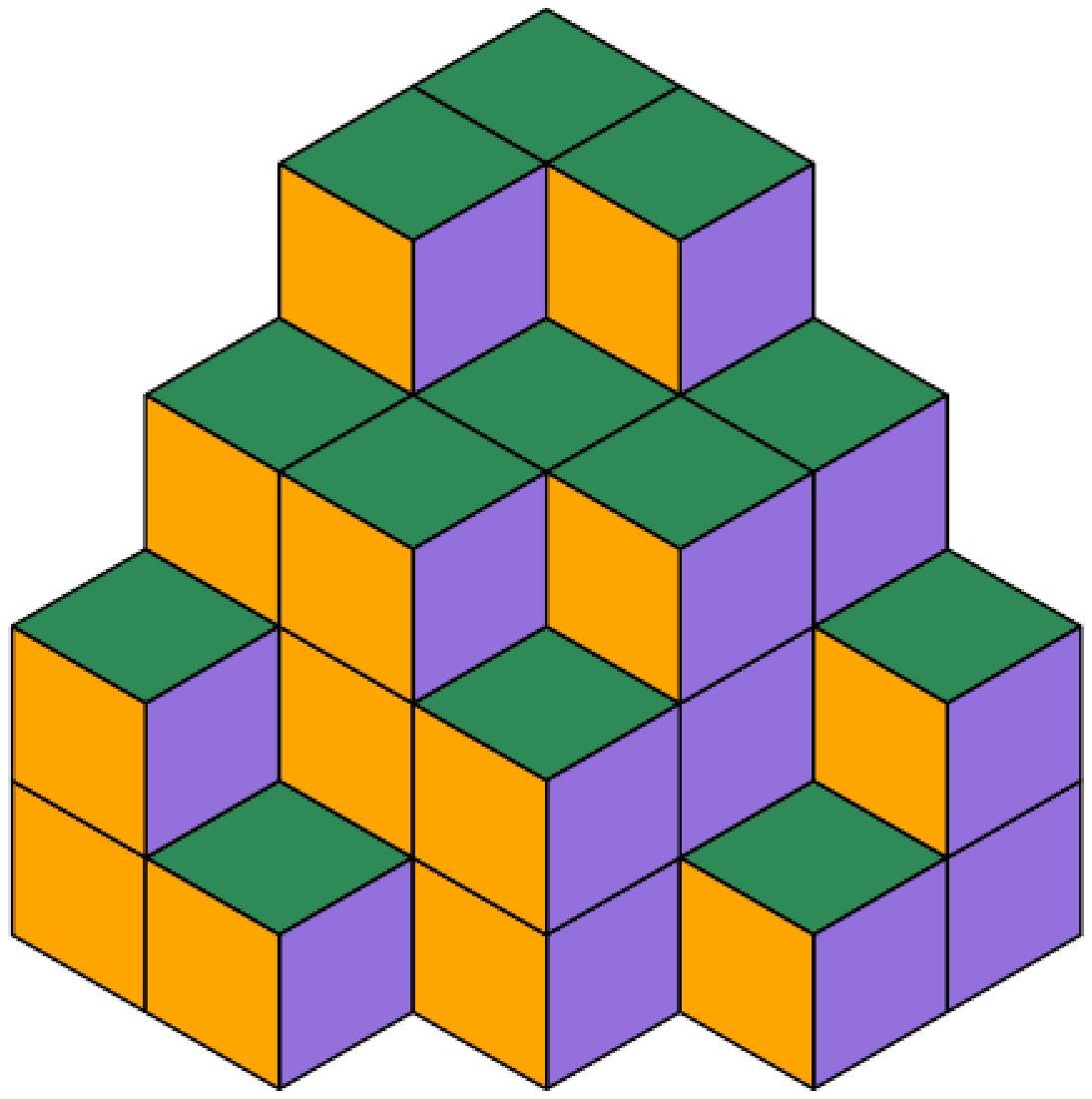}
\end{center}
\end{example}

\section{Monomial Ideals}
Fix a positive integer $n$ and let $K$ be a field of characteristic zero. We will often refer to \textit{monomials} by their multi-exponent notation $\textbf{x}^{\alpha}:=x_1^{a_1}\cdots x_d^{a_d}\in K[x_1,\dots,x_d]$. In this notation, let $deg(\alpha)=\sum_{i=1}^da_i$. A \textit{monomial ideal} is an ideal $I\subset K[x_1,\dots,x_d]$ which is generated by monomials. Every monomial ideal $I\subset K[x_1,\dots,x_d]$ has a unique minimal subset of monomial generators $G(I)$.

Elements $(a_{i,j})\in GL_d(K)$ induce a linear automorphism on $K[x_1,\dots,x_d]$
\begin{equation*}
f(x_1,\dots,x_d)\mapsto f\big(\sum_{i=1}^d a_{i,1}x_i,\dots,\sum_{i=1}^d a_{i,d}x_i\big)
\end{equation*}

The \textit{Borel group} is the subgroup of $GL_d(K)$ consisting of upper triangular matrices. The ideals which are fixed by the action of the Borel group on the variables $x_1,\dots,x_d$ are called \textit{Borel-fixed ideals}.

\begin{definition}
Let $I\subset K[x_1,\dots,x_d]$ be a monomial ideal. If $m\frac{x_i}{x_j}\in I$ for all $m\in I$ with $x_j\mid m$ and $i<j$, then $I$ is called \textit{strongly stable}.
\end{definition}

We will refer to the defining property of strongly stable ideals as the \textit{variable exchange condition}.

\begin{proposition}
Let $I$ be a monomial ideal. If $char(K)=0$, then $I$ is strongly stable iff $I$ is Borel-fixed.
\end{proposition}

Note that we are implicitly using the variable order $x_1>x_2>\cdots >x_d$. By shuffling this order, we would get different (but equivalent) sets of strongly stable ideals. The next proposition shows that we can check if $I$ is strongly stable by checking the variable exchange condition on generators of $I$.

\begin{proposition}
Let $I\subset K[x_1,\dots,x_d]$ be a monomial ideal. Suppose that for all $m\in G(I)$ and all integers $1\leq i<j\leq d$ such that $x_j\mid m$, we have $m\frac{x_i}{x_j}\in I$. Then $I$ is strongly stable.
\end{proposition}

See \cite{herzog2011monomial} for proofs of the previous two propositions.

\begin{example}
The ideal $I=(x^4,x^3y,x^2y^3,xy^4,y^7)$ is strongly stable, because $x^3y\frac{x}{y}=x^4\in I$, $x^2y^3\frac{x}{y}=x^3y^2\in I$, $xy^4\frac{x}{y}=x^2y^3\in I$, and $y^7\frac{x}{y}=xy^6\in I$.
\end{example}

Given a set of monomials $A=\{\textbf{x}^{\alpha_1},\dots,\textbf{x}^{\alpha_r}\}$, we can consider the minimal subset of $A$ which generates the same ideal. We will use the notation $min(A)$ to refer to this minimal subset.

\begin{definition}
Let $m\in K[x_1,\dots,x_d]$ be a monomial. A \textit{Borel move} is an operation that sends $m$ to a monomial $m\frac{x_{i_1}}{x_{j_1}}\cdots\frac{x_{i_s}}{x_{j_s}}$, where $i_t<j_t$ and $m$ is divisible by $x_{j_t}$ for all $t$.
\end{definition}

\begin{definition}
A monomial ideal $I\subset K[x_1,\dots,x_d]$ is a \textit{strongly stable ideal} if it is closed under Borel moves.
\end{definition}

\begin{definition}
Let $A$ be a subset of monomials. Define $Borel(A)$ to be the smallest strongly stable ideal containing $A$. We call the monomials in $A$ \textit{Borel generators} of $Borel(A)$.
\end{definition}

\begin{proposition}
Every strongly stable ideal $I$ has a unique minimal set of Borel generators. Refer to this set as $Bgens(I)$.
\end{proposition}

\begin{proposition}
Suppose $I$ is a strongly stable ideal and $m\in I$. Then $m\in Bgens(I)$ iff for all $x_q$ dividing $m$, $\frac{m}{x_q}\notin I$ and $m\frac{x_{q+1}}{x_q}\notin I$.
\end{proposition}

Definitions 3.5, 3.6, 3.7 and Propositions 3.8 and 3.9 are due to \cite{francisco2011borel}. Note that definitions 3.5 and 3.6 are equivalent to definition 3.1. Also, in \cite{francisco2011borel}, they refer to strongly stable ideals as \textit{Borel ideals}. We have chosen to use the term strongly stable in this paper.

\begin{example}
For the ideal $I=(x^4,x^3y,x^2y^3,xy^4,y^7)$,
\begin{equation*}
Bgens(I)=\{x^3y,xy^4,y^7\}
\end{equation*}
\end{example}

\begin{definition}
A monomial ideal $I$ is \textit{symmetric} if it is closed under the action of the symmetric group $S_d$ on the variables.
\end{definition}

\begin{proposition}
A monomial ideal $I$ is symmetric iff $G(I)$ is closed under the action of the symmetric group $S_d$ on the variables.
\end{proposition}

\begin{proof}
($\Rightarrow$) Let $g\in G(I)$ and $\pi\in S_d$. Since $I$ is symmetric, $\pi(g)\in I$. To see that $\pi(g)\in G(I)$, suppose not. Then there exists $j$ so that $\frac{\pi(g)}{x_j}\in I$. It follows that $\pi^{-1}(\frac{\pi(g)}{x_j})=\frac{g}{x_{\pi^{-1}(j)}}$ so $g\notin G(I)$. Contradiction.

($\Leftarrow$) Let $u\in I$. Then there exists some $g\in G(I)$ so that $g\mid u$. It follows that $\pi(g)\mid\pi(u)$ for all $\pi\in S_d$, so $\pi(u)\in I$.
\end{proof}

Given a set of monomials $A=\{\textbf{x}^{\alpha_1},\dots,\textbf{x}^{\alpha_r}\}$, we can \textit{symmetrize} this set of monomials by letting $S_d$ act on the variables. Denote this set by 
\begin{equation*}
sym(A):=\{\textbf{x}^{\pi(\alpha_j)}\mid \textbf{x}^{\alpha_j}\in A;\pi\in S_d\}
\end{equation*}

When $A=\{\textbf{x}^\alpha\}$ consists of a single monomial, we refer to $sym(A)$ as the \textit{orbit} of $\textbf{x}^\alpha$.

\begin{definition}
A monomial $\textbf{x}^\alpha$ is a \textit{pure power} of $x_j$ if $\alpha=(a_1,\dots,a_j,\dots,a_d)$ with $a_i=0$ for $i\neq j$.
\end{definition}

\begin{definition}
An ideal $I\subset K[x_1,\dots,x_d]$ is \textit{Artinian} if the Krull dimension of $R/I$ is zero.
\end{definition}

Note that a monomial ideal is Artinian if and only if $I$ contains a pure power of $x_j$ for $1\leq j\leq d$. Denote the set of Artinian ideals with the largest degree of any pure power equal to $n$ by $\mathcal{A}_d(n)$. Denote the set of Artinian strongly stable ideals with pure power $x_d^n\in G(I)$ by 
\begin{equation*}
\mathcal{B}_d(n):=\{I\in\mathcal{A}_d(n)\mid \text{$I$ is strongly stable}\} 
\end{equation*}
Similarly, use
\begin{equation*}
\mathcal{T}_d(n):=\{I\in\mathcal{A}_d(n)\mid \text{$I$ is symmetric}\} 
\end{equation*}
to denote the set of Artinian symmetric monomial ideals with pure power $x_d^n\in G(I)$.

\begin{proposition}
Let $I\in\mathcal{B}_d(n)$ with pure power $x_d^n\in G(I)$. Then for every pure power $x_j^{k_j}\in G(I)$, we have $k_j\leq n$.
\end{proposition}

\begin{proof}
We can use Borel moves to get $x_j^n\in I$ for all $j$. So every pure power $x_j^{k_j}\in G(I)$ must satisfy $k_j\leq n$.
\end{proof}

\begin{proposition}
Let $I\in\mathcal{T}_d(n)$ with pure power $x_d^n\in G(I)$. Then for every pure power $x_j^{k_j}\in G(I)$, we have $k_j=n$.
\end{proposition}

\begin{proof}
This follows from Proposition 3.12.
\end{proof}

\section{Partitions Correspond to Artinian Monomial Ideals}

\begin{proposition}
If $I\in\mathcal{A}_d(n)$, then $P=\{\alpha\mid\textnormal{\textbf{x}}^\alpha\notin I\}\in\mathcal{P}_d(n)$.
\end{proposition}

\begin{proof}
First, we need to show that $P=\{\alpha\mid\textbf{x}^\alpha\notin I\}$ is a partition. Just notice that if $\alpha=(a_1,\dots,a_d)\in P$, then $\textbf{x}^\alpha\notin I$. It follows that $\textbf{x}^{(a_1,\dots,a_j-1,\dots,a_d)}\notin I$ so $(a_1,\dots,a_j-1,\dots,a_d)\in P$. Since $I$ is Artinian, $P$ is finite.

For every cell $\gamma=(c_1,\dots,c_d)\in P$, we have $c_i<n$, because $I$ contains a pure power of every variable with degree less than or equal to $n$. Since $x_d^n\in G(I)$, $(0,0,\dots,0,n-1)\in P$.
\end{proof}

\begin{proposition}
If $P\in\mathcal{P}_d(n)$, then $\{\textnormal{\textbf{x}}^\alpha\mid\alpha\notin P\}\in\mathcal{A}_d(n)$.
\end{proposition}

\begin{proof}
Let $I=\{\textbf{x}^\alpha\mid\alpha\notin P\}$. It is clear that $I$ is a monomial ideal. Since $(0,0,\dots,n,\dots,0,0)\notin P$, $x_j^n\in I$ for $1\leq j\leq d$.

There is some cell $\gamma=(c_1,\dots,c_j,\dots,c_d)\in P$ with $c_j=n-1$ for some $1\leq j\leq d$. By the definition of $d$-dimensional partition, $(0,0,\dots,n-1,\dots,0)\in P$, so $x_j^n\in G(I)$.
\end{proof}

The two propositions above allow us to define a map $\varphi:\mathcal{A}_d(n)\rightarrow\mathcal{P}_d(n)$ by
\begin{equation*}
\varphi(I)=\{\alpha\in\mathbb{N}^d\mid\textbf{x}^\alpha\notin I\}
\end{equation*}
with inverse $\varphi^{-1}:\mathcal{P}_d(n)\rightarrow\mathcal{A}_d(n)$ given by $\varphi^{-1}(P)=\{\textbf{x}^\alpha\mid\alpha\notin P\}$.

\begin{example}
Let $I=(x^4,x^2y,xy^2,y^3)\in\mathcal{A}_2(4)$. Then $\varphi(I)$ is the partition from Example 2.2 and is shown below with $\bullet$ used to represent generators of $I$.
\begin{center}
\begin{ytableau}
    \none & \none[\bullet] \\
    \none & & \none[\bullet]\\
    \none & & & \none[\bullet]\\
    \none & & & & & \none[\bullet]\\
\end{ytableau}
\end{center}
\end{example}

\begin{example}
Let $I=(x^3,x^2y,y^3,x^2z,xyz,y^2z,z^2)\in\mathcal{A}_3(3)$. Then $\varphi(I)$ is the plane partition from Example 2.3 shown below.
\begin{center}
\includegraphics[scale=0.3]{pp}
\end{center}
\end{example}

We will show that we can restrict $\varphi$ to a bijection between strongly stable ideals and strongly stable partitions and between symmetric monomial ideals and totally symmetric partitions.

\begin{proposition}
A monomial ideal $I$ is strongly stable iff $\varphi(I)$ is a strongly stable partition.
\end{proposition}

\begin{proof}
($\Rightarrow$) Let $\beta\in\varphi(I)$ so that $\textbf{x}^\beta\notin I$ and consider $H(\beta)=(h_1,\dots,h_d)$. Fix $1\leq i<j\leq d$. We have $(b_1,\dots,b_i,\dots,b_j+h_j+1,\dots,b_d)\notin\varphi(I)$, and we can use the fact that $I$ is strongly stable to get 
\begin{equation*}
\textbf{x}^{(b_1,\dots,b_i,\dots,b_j+h_j+1,\dots,b_d)}\in I\implies\textbf{x}^{(b_1,\dots,b_i+h_j+1,\dots,b_j,\dots,b_d)}\in I
\end{equation*}
so $(b_1,\dots,b_i+h_j+1,\dots,b_j,\dots,b_d)\notin P$. Therefore, $h_i<h_j+1\implies h_i\leq h_j$.

($\Leftarrow$) Let $\textbf{x}^\alpha\in G(I)$ such that $x_j\mid\textbf{x}^\alpha$. Then $\frac{\textbf{x}^\alpha}{x_j}=\textbf{x}^{(a_1,\dots,a_j-1,\dots,a_d)}\notin I$, so $\tilde{\alpha}=(a_1,\dots,a_j-1,\dots,a_d)\in\varphi(I)$. If $H(\tilde{\alpha})=(h_1,\dots,h_d)$, then $h_j=1$ because $\alpha\notin P$. Since $\varphi(I)$ is a strongly stable partition, $h_i\leq h_j=1$ for $i<j$. In particular, $(a_1,\dots,a_i+1,\dots,a_j-1,\dots,a_d)\notin I$. So $\textbf{x}^\alpha\frac{x_i}{x_j}\in I$.
\end{proof}

\begin{proposition}
A monomial ideal $I$ is symmetric iff $\varphi(I)$ is a totally symmetric partition.
\end{proposition}

\begin{proof}
($\Rightarrow$) Let $\beta\in\varphi(I)$ and let $\pi\in S_d$. Since $\textbf{x}^\beta\notin I$, $\textbf{x}^{\pi(\beta)}\notin I$, so $\pi(\beta)\in\varphi(I)$.\\

($\Leftarrow$) Let $\textbf{x}^\alpha\in I$. Then $\alpha\notin\varphi(I)\implies \pi(\alpha)\notin\varphi(I)$, so $\textbf{x}^{\pi(\alpha)}\in I$.
\end{proof}

In this section, we have shown that $\varphi\mid_{\mathcal{B}_d(n)}:\mathcal{B}_d(n)\rightarrow\tilde{\mathcal{B}}_d(n)$ and $\varphi\mid_{\mathcal{T}_d(n)}:\mathcal{T}_d(n)\rightarrow\tilde{\mathcal{T}}_d(n)$ are bijections. In the next section, we show a bijection $\mathcal{B}_d(n)\rightarrow\mathcal{T}_d(n)$.

\section{Bijection Between Monomial Ideals}

Denote the set of all monomials in $K[x_1,\dots,x_d]$ by $M_d$ and the set of monomials with weakly increasing exponent vector by $F_d$. Then use $\mathcal{F}_d(n)$ to denote the set of minimal subsets of $F_d$ which contain $x_d^n$ and no coordinate of any exponent vector exceeding $n$. Define a map $\psi:M_d\rightarrow F_d$ by
\begin{equation*}
\psi(\textbf{x}^{(a_1,a_2,\dots,a_d)})=\textbf{x}^{(a_1,a_1+a_2,a_1+a_2+a_3,\dots,a_1+a_2+\cdots+a_d)}
\end{equation*}
and notice that $\psi^{-1}:F_d\rightarrow M_d$ is given by
\begin{equation*}
\psi^{-1}(\textbf{x}^{(a_1,\dots,a_d)})=\textbf{x}^{(a_1,a_2-a_1,a_3-a_2,\dots,a_d-a_{d-1})}
\end{equation*}

\begin{proposition}
$\psi(m\frac{x_{q+1}}{x_q})=\frac{\psi(m)}{x_q}$
\end{proposition}

\begin{proof}
\begin{align*}
\psi(m\frac{x_{q+1}}{x_q})&=\psi(\textbf{x}^{(a_1,\dots,a_q-1,a_{q+1}+1,\dots,a_d)})\\
&=\textbf{x}^{(a_1,a_2-a_1,a_3-a_2,\dots,a_q-a_{q-1}-1,a_{q+1}-a_q,a_{q+2}-a_{q+1},\dots,a_d-a_{d-1})}\\
&=\frac{\psi(m)}{x_q}
\end{align*}
\end{proof}

For a subset of monomials $A=\{\textbf{x}^{\alpha_1},\dots,\textbf{x}^{\alpha_r}\}$, we will use $\psi(A):=\{\psi(\textbf{x}^{\alpha_1}),\dots,\psi(\textbf{x}^{\alpha_r})\}$. For an ideal $I$, we will use $\psi(I)$ to refer to the ideal generated by $\psi(G(I))$. 

\begin{proposition}
For a strongly stable ideal $I$, $m\in I\iff\psi(m)\in\psi(I)$.
\end{proposition}

\begin{proof}
($\Rightarrow$) If $m\in I$, then there exists $g\in G(I)$ so that $g\mid m$. If $g=\textbf{x}^{(c_1,\dots,c_d)}$ and $m=\textbf{x}^{(a_1,\dots,a_d)}$, then $c_i\leq a_i$ for $1\leq i\leq d$. It follows that $c_1+\cdots+c_i\leq a_1+\cdots+a_i$ for all $i$. So $\psi(g)\mid\psi(m)$.\\
($\Leftarrow$) Assume $\psi(m)\in\psi(I)$ and let $k_i:=a_1+\cdots +a_i-(c_1+\cdots+ c_i)$ so that
\begin{equation*}
\textbf{x}^{(k_1,\dots,k_d)}\psi(g)=\psi(m)
\end{equation*}
Note that $k_{i}=a_{i}-c_{i}+k_{i-1}$ for $i=2,\dots,d$.
Then we have
\begin{align*}
\tilde{g}&:=g\frac{x_1^{k_1}}{x_2^{k_1}}\frac{x_2^{k_2}}{x_3^{k_2}}\cdots\frac{x_{d-1}^{k_{d-1}}}{x_d^{k_{d-1}}}\\
&=g\frac{x_1^{k_1-k_2}}{x_2^{k_1-k_2}}\frac{x_1^{k_2-k_3}}{x_3^{k_2-k_3}}\cdots\frac{x_1^{k_{d-2}-k_{d-1}}}{x_{d-1}^{k_{d-2}-k_{d-1}}}\frac{x_1^{k_{d-1}}}{x_d^{k_{d-1}}}\\
&=\textbf{x}^{(c_1+k_1,c_2+k_2-k_1,c_3+k_3-k_2,\dots,c_{d-1}+k_{d-1}-k_{d-2},c_d-k_{d-1})}\\
&=\textbf{x}^{(a_1,a_2,\dots,a_{d-1},c_d-k_{d-1})}
\end{align*}
since $c_i+k_i-k_{i-1}=c_i+(a_i-c_i+k_{i-1})-k_{i-1}=a_i$.
For the last coordinate, we have
\begin{align*}
c_d-k_{d-1}&=c_d-(a_1+\cdots+a_{d-1}-(c_1+\cdots+c_{d-1}))\\
&=c_1+\cdots+c_d - (a_1+\cdots+a_{d-1})\\
&\leq a_d
\end{align*}
We have shown that $\tilde{g}\mid m$.
Now, we claim that $\tilde{g}\in I$ because it is a Borel move from $g$. To see this, notice that $k_{i-1}-k_{i}=c_i-a_i$ for $i=2,\dots,d$ and $k_{d-1}=a_1+\cdots+a_{d-1}-(c_1+\cdots+c_{d-1})\leq c_d$. So $g$ is divisible by each of the denominators in the second line of $\tilde{g}$.
\end{proof}

\begin{proposition}
If $I$ is a strongly stable ideal, then $Bgens(I)=\psi^{-1}(min(\psi(G(I))))$.
\end{proposition}

\begin{proof}
($\subseteq$) Let $m\in Bgens(I)\subset G(I)$. For any $x_q\mid\psi(m)$, $\frac{\psi(m)}{x_q}=\psi(m\frac{x_{q+1}}{x_q})\notin\psi(I)$ because $m\frac{x_{q+1}}{x_q}\notin I$. Therefore, $\psi(m)\in min(\psi(G(I)))$.

($\supseteq$) Let $m\in \psi^{-1}(min(\psi(G(I))))$. Then $\psi(m)$ is minimal, so $\frac{\psi(m)}{x_q}=\psi(m\frac{x_{q+1}}{x_q})\notin\psi(I)$. It follows that $m\frac{x_{q+1}}{x_q}\notin I$.
\end{proof}

Note that this gives an algorithm to compute $Bgens(I)$ given a generating set $G(I)$. As a result of Proposition 3.15, observe that for any $I\in\mathcal{B}_d(n)$, we have
\begin{equation*}
\textbf{x}^\alpha\in G(I)\implies deg(\alpha)\leq n.
\end{equation*}
It follows that every coordinate of $\psi(\textbf{x}^\alpha)$ is less than or equal to $n$. This observation combined with the previous proposition allows us to define a map
\begin{align*}
\Lambda:\mathcal{B}_d(n)&\rightarrow\mathcal{F}_d(n)\\
I&\mapsto \psi(Bgens(I)).
\end{align*}

\begin{proposition}
$\Lambda:\mathcal{B}_d(n)\rightarrow\mathcal{F}_d(n)$ is a bijection.
\end{proposition}

\begin{proof}
We claim that $\Lambda^{-1}:\mathcal{F}_d(n)\rightarrow\mathcal{B}_d(n)$ is given by $S\mapsto Borel(\psi^{-1}(S))$. Since $x_d^n\in S$, $\psi^{-1}(x_d^n)=x_d^n\in\psi^{-1}(S)$. There are no other pure powers of $x_d$ in $\psi^{-1}(S)$, so $x_d^n\in G(I)$. It follows that $Borel(\psi^{-1}(S))\in\mathcal{B}_d(n)$ and
\begin{align*}
(\Lambda^{-1}\circ\Lambda)(I)&=Borel(\psi^{-1}(\psi(Bgens(I))))\\
&=I.
\end{align*} 
\end{proof}

\begin{proposition}
If $I\subset K[x_1,\dots,x_d]$ is a symmetric monomial ideal, then $sym(G(I)\cap F_d)=G(I)$.
\end{proposition}

\begin{proof}
($\subseteq$) If $u\in sym(G(I)\cap F_d)$, then $u=\pi(m)$ for some $m\in G(I)\cap F_d$ and some $\pi\in S_d$. Since $G(I)$ is closed under operations of $S_d$, $u\in G(I)$.\\
($\supseteq$) If $u\in G(I)$, then there exists some $\pi\in S_d$ so that $\pi(u)\in G(I)\cap F_d$, so $u\in sym(G(I)\cap F_d)$.
\end{proof}

For an ideal $I\in\mathcal{T}_d(n)$, define a map $\Omega:\mathcal{F}_d(n)\rightarrow\mathcal{T}_d(n)$ by
\begin{equation*}
\Omega(A)=ideal(sym(A))
\end{equation*}

\begin{proposition}
$\Omega:\mathcal{F}_d(n)\rightarrow\mathcal{T}_d(n)$ is a bijection
\end{proposition}

\begin{proof}
We claim that the inverse $\Omega^{-1}:\mathcal{T}_d(n)\rightarrow\mathcal{F}_d(n)$ is given by $\Omega^{-1}(I)=G(I)\cap F_d$. Notice that for $I\in\mathcal{T}_d(n)$, we have $x_d^n\in G(I)$, so $x_d^n\in G(I)\cap F_d$. It follows that $\Omega^{-1}(I)\in\mathcal{F}_d(n)$ and
\begin{align*}
\Omega(\Omega^{-1}(I))&=\Omega(G(I)\cap F_d)\\
&=ideal(sym(G(I)\cap F_d))\\
&=I.
\end{align*}
\end{proof}

\begin{theorem}
There is a bijection between
$d$-dimensional strongly stable partitions and $d$-dimensional totally symmetric partitions which preserves the side length $n$ of the minimal $d$-dimensional box containing the partitions.
\end{theorem}

\begin{proof}
We have the diagram of bijections below. The vertical bijections were given in section 4 and the horizontal map $\Omega\circ\Lambda$ was shown to be a bijection by Propositions 5.4 and 5.6.

\[ \begin{tikzcd}
\mathcal{B}_d(n) \arrow{r}{\Omega\circ\Lambda} \arrow[swap]{d}{\varphi} & \mathcal{T}_d(n) \arrow{d}{\varphi} \\%
\tilde{\mathcal{B}}_d(n) \arrow{r}& \tilde{\mathcal{T}}_d(n)
\end{tikzcd}
\]
\end{proof}

\begin{example}
Consider the strongly stable ideal $I=(x^4,x^3y,x^2y^3,xy^4,y^7)\in\mathcal{B}_2(7)$ from Example 3.10. The corresponding partition $\varphi(I)\in\tilde{\mathcal{B}}_2(7)$ from Example 2.6 is shown on the left below. Minimal generators of $I$ which are not in $Bgens(I)$ are represented with $\bullet$. Elements of $Bgens(I)$ are represented by $*$. The middle diagram depicts $\psi(I)$ and illustrates that $\psi(Bgens(I))=min(\psi(G(I)))$. The diagram on the right is the corresponding totally symmetric partition.

\begin{center}
\includegraphics[scale=0.9]{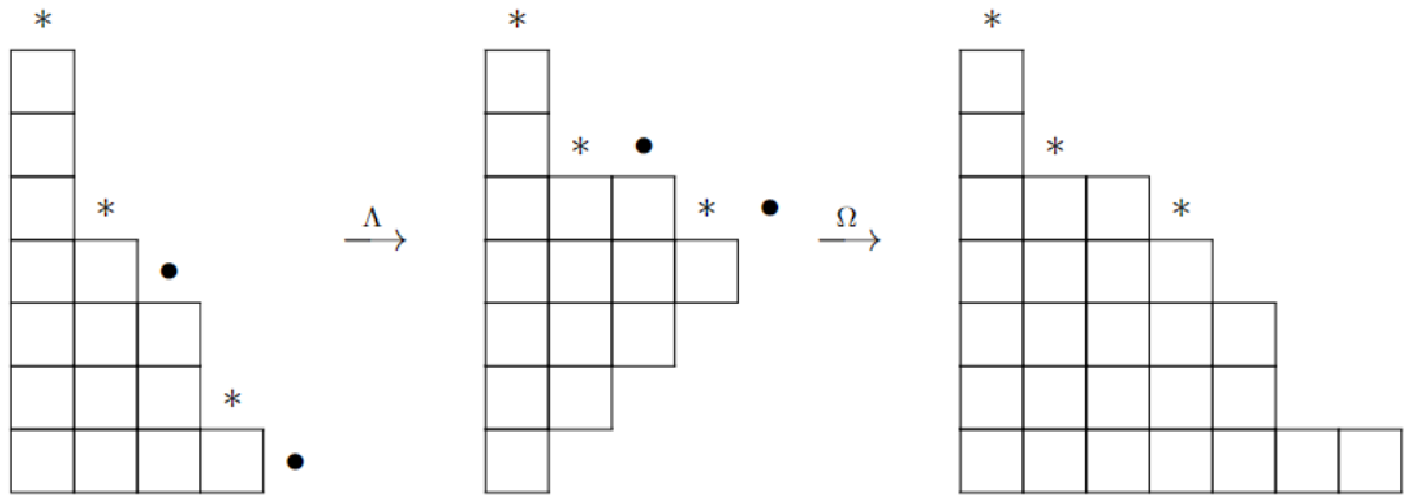}
\end{center}

\begin{remark}
For any $\textbf{x}^\alpha\notin I$, we have $\psi(\textbf{x}^\alpha)\notin\Lambda(I)$. Since $\psi(\textbf{x}^\alpha)$ is the unique representative of its orbit, the number of orbits of monomials not in $\Omega(\Lambda(I))$ is exactly the total number of monomials in the original ideal $I$. In partition language, this means that the number of cells in the strongly stable partition is exactly the number of orbits in the totally symmetric partition. In the example above, there are 15 cells in the strongly stable partition on the left which correspond to 15 orbits in the totally symmetric partition on the right.
\end{remark}

\end{example}

\begin{example}
Consider the strongly stable partition $P\in\tilde{\mathcal{B}}_3(4)$ from Example 2.7 shown below.
\begin{center}
\includegraphics[scale=0.4]{ss}
\end{center}
One can check that $I=(x^2,xy^2,y^3,xyz,y^2z,xz^2,yz^3,z^4)\in\mathcal{B}_3(4)$ is the corresponding strongly stable ideal ($\varphi(I)=P$) and that $Bgens(I)=\{x^2,xz^2,y^2z,z^4\}$. To compute the totally symmetric partition corresponding to $P$, we first compute
\begin{align*}
    \Omega(\Lambda(I))&=\Omega(\psi(Bgens(I)))\\
    &=\Omega(\{x^2y^2z^2,xyz^3,y^2z^3,z^4\})\\
    &=ideal(sym(\{x^2y^2z^2,xyz^3,y^2z^3,z^4\}))
\end{align*}
\end{example}
Finally, notice that $\varphi(ideal(sym(\{x^2y^2z^2,xyz^3,y^2z^3,z^4\})))\in\tilde{\mathcal{T}}_3(4)$ is the partition shown below. 
\begin{center}
\includegraphics[scale=0.5]{ts}
\end{center}

\section{Enumerations}
In this last section, we will explore the enumeration of strongly stable partitions (equivalently totally symmetric partitions) in boxes. Let
\begin{equation*}
B_d(n):=\sum_{k=0}^n|\tilde{\mathcal{B}}_d(k)|
\end{equation*}
denote the number of $d$-dimensional strongly stable partitions which fit in a box of side length $n$. Similarly, use
\begin{equation*}
T_d(n):=\sum_{k=0}^n|\tilde{\mathcal{T}}_d(k)|
\end{equation*}
By convention, we count the empty partition. As a result, we have $B_d(0)=T_d(0)=1$ since the empty partition fits inside a box of side length $0$. As a result of Theorem 1,
\begin{equation*}
B_d(n)=T_d(n)
\end{equation*}
for positive integers $d,n$.

In \cite{tspib}, Hawkes shows a bijection between $d$-dimensional totally symmetric partitions inside a box of side length $n$ and $(n-1)$-dimensional totally symmetric partitions inside a box of side length $d+1$. As a result, $T_d(n)=T_{n-1}(d+1)$ for $n\geq 2$. Combining this with the formula above, we have
\begin{equation*}
B_d(n)=B_{n-1}(d+1)
\end{equation*}
for positive integers $d,n$.

Lastly, we consider some formulae for $B_d(n)$ when $d$ is fixed. When $d=1$, partitions are simply nonnegative integers, and every $1$-dimensional partition is (trivially) strongly stable.
\begin{equation*}
B_1(n)=T_1(n)=n+1
\end{equation*}

The number of $2$-dimensional strongly stable partitions which fit inside a box of side length $n$ (strict integer partitions with largest part $n$) is given by
\begin{equation*}
B_2(n)=T_2(n)=2^n,
\end{equation*}
because for $1\leq i\leq n$, we have only the choice of whether or not to include $i$ in the integer partition. 

In \cite{stembridge1995enumeration}, Stembridge proved that the number of totally symmetric plane partitions which fit inside a box of side length $n$ is given by the product formula
\begin{equation*}
T_3(n)=\prod_{1\leq i\leq j\leq k\leq n}\frac{i+j+k-1}{i+j+k-2}
\end{equation*}
and so we have a formula for $B_3(n)$.

In fact, in the case $d=3$, an even stronger result has been shown for totally symmetric plane partitions. The $q$-TSPP formula 
\begin{equation*}
\sum_{P\in\tilde{\mathcal{T}}_3(n)}q^{\abs{P/S_3}}=\prod_{1\leq i\leq j\leq k\leq n}\frac{1-q^{i+j+k-1}}{1-q^{i+j+k-2}}
\end{equation*}
was shown to be the orbit-counting generating function for totally symmetric plane partitions fitting inside an $n\times n\times n$ box in \cite{koutschan2011proof}. Remark 5.8 implies that this same formula is a \textit{cell}-counting generating function for strongly stable plane partitions.
 
There is no known formula for $B_d(n)=T_d(n)$ for $d\geq 4$ \cite{ts_solid_count}.

\newpage
\printbibliography

@article{francisco2011borel,
  title={Borel generators},
  author={Francisco, Christopher A and Mermin, Jeffrey and Schweig, Jay},
  journal={Journal of Algebra},
  volume={332},
  number={1},
  pages={522--542},
  year={2011},
  publisher={Elsevier}
}

@article{stembridge1995enumeration,
  title={The enumeration of totally symmetric plane partitions},
  author={Stembridge, John R},
  journal={Advances in Mathematics},
  volume={111},
  number={2},
  pages={227--243},
  year={1995},
  publisher={Elsevier}
}

@misc{tspib,
  title        = "{Totally Symmetric Partitions in Boxes}",
  author       = "{Hawkes, Graham}",
  howpublished = "\url{https://oeis.org/A005157/a005157_2.pdf}",
  year         = 2014,
  note         = "Accessed: 2023-01-22"
}

@inbook{herzog2011monomial,
  title={Monomial ideals},
  author={Herzog, J{\"u}rgen and Hibi, Takayuki},
  booktitle={Monomial Ideals},
  pages={57},
  year={2011},
  publisher={Springer}
}

@misc{ts_solid_count,
  title        = "{Number of totally symmetric solid partitions which fit in an n X n X n X n box}",
  author       = "{Hawkes, Graham}",
  howpublished = "\url{https://oeis.org/A236691}",
  year         = 2014,
  note         = "Accessed: 2023-01-22"
}

@article{koutschan2011proof,
  title={Proof of George Andrews’s and David Robbins’ sq-TSPP conjecture},
  author={Koutschan, Christoph and Kauers, Manuel and Zeilberger, Doron},
  journal={Proceedings of the National Academy of Sciences},
  volume={108},
  number={6},
  pages={2196--2199},
  year={2011},
  publisher={National Acad Sciences}
}
\end{document}